\documentclass[12pt,a4paper]{amsart}
\usepackage{tikz}
\usepackage{pdflscape}
\usepackage{amsmath}
\usepackage{amsfonts}
\usepackage{amssymb}
\usepackage{mathtools}
\usepackage[all]{xy}
\usepackage{color}
\usetikzlibrary{arrows}
\usepackage{float}
\usepackage[shortlabels]{enumitem}
\usepackage{ifpdf}
\ifpdf
  \usepackage[colorlinks,
            linkcolor=magenta,       
            anchorcolor=magenta,     
            citecolor=magenta,       
            final,
            hyperindex]{hyperref}
\else
  \usepackage[colorlinks,
            linkcolor=magenta,       
            anchorcolor=magenta,     
            citecolor=magenta,       
            final,
            hyperindex,driverfallback=hypertex]{hyperref}
\fi
\newtheorem{theorem}{Theorem}[section]
\newtheorem{lemma}[theorem]{Lemma}
\newtheorem{coro}[theorem]{Corollary}

\theoremstyle{definition}
\newtheorem{defn}[theorem]{Definition}
\newtheorem{remark}[theorem]{Remark}

\makeatletter \@addtoreset{equation}{section} \makeatother
\allowdisplaybreaks

\newcommand{\delete}[1]{}

\newcommand{\lambdaa}[2]{\lambda_{#1}{#2}}

\newcommand{\rhoo}[2]{\rho_{#1}{#2}}

\begin{document}
\title[Left semibraces, near left semibraces and the Yang-Baxter equation]{Left semibraces, near left semibraces and the Yang-Baxter equation}


\author{Shoufeng Wang}
\address{School of Mathematics, Yunnan Normal University, Kunming, Yunnan 650500, China}
\email{wsf1004@163.com}


\begin{abstract}  A triple $(S, +, \cdot)$ is called a {\em left semibrace} if $(S, +)$  is a semigroup, $(S, \cdot)$ is a group with identity $e$ and $x(y+z)=(xy)+(x(x^{-1}+z))$ for all $x, y, z\in S$,  where $x^{-1}$ is the inverse of $x$ in the group $(S, \cdot)$. A left semibrace $(S,+,\cdot)$ is called {\em strong} if $x+y(e+z)=x+yz$ for all $x, y, z\in S$.  Left semibraces and strong left semibraces are investigated extensively in literature. However, the structure of additive semigroups of general left semibraces still remains mysterious.
 In this note, as generalizations of near left braces, we introduce  {\em near left semibraces} as follows.
A quadruple $(S,+, \cdot, \mu)$ is called a {\em near left semibrace} if $(S, +)$  is a semigroup, $(S, \cdot)$ is a group, $\mu: S\to S$ is a map and $x(y+z)=(xy)+\mu(x)+xz$ for all $x, y, z\in S$.  We first show that the additive semigroups of both left semibrace and near left semibraces are rectangular groups  and obtain some new characterizations of  strong left semibraces. In particular, we  prove that a strong left semibrace can induce a near left semibrace, and vice versa. As a consequence,   near left semibraces can provide set-theoretical solutions for the Yang-Baxter equation.   Next, we obtain a structure theorem for all left semibraces by the generalized matched products of right zero left semibraces and right cancellative left semibraces. Finally,  we consider a new map associated to a  near left semibrace and give a sufficient and necessary condition under which such a map forms a set-theoretic  solution of the Yang-Baxter equation.
Our result improve and enrich some results obtained by  Jespers and Van Antwerpen in [Forum Math. 31 (2019) 241--263], by  Catino,   Colazzo and Stefanelli in [Mediterr. J. Math. 17 (2020) 58] and by Catino, Mazzotta and Stefanelli in [J. Algebra  573 (2021) 576--619].

\end{abstract}
\makeatletter
\@namedef{subjclassname@2020}{\textup{2020} Mathematics Subject Classification}
\makeatother

\maketitle



\vspace{-0.8cm}
\section{Introduction}
The quantum Yang-Baxter equation first appeared in theoretical physics in Yang \cite{Yang} and  in statistical mechanics  in Baxter \cite{Baxter} independently. Now,  the Yang-Baxter equation has many applications in the areas of mathematics and mathematical physics such as knot theory, quantum computation, quantum group theory and so on. Finding all solutions of the Yang-Baxter equation is presently impossible, so in 1992, Drinfeld \cite{Drinfeld} posed the question of finding all the so called set-theoretic  solutions of the Yang-Baxter equation. Let $S$ be a non-empty set.  If a map $r: S\times S\longrightarrow S\times S$
satisfies \begin{equation}\label{Y-B}
(r\times {\rm id}_{S})({\rm id}_{S}\times r)(r\times {\rm id}_{S})=({\rm id}_{S}\times r)(r\times
{\rm id}_{S})({\rm id}_{S}\times r)
\end{equation} in the set of maps from $S\times S\times S$ to $S\times S\times S$, where  ${\rm id}_{S}$ is the identity map on $S$, then  $r$ is called a {\em set-theoretic  solution} of the Yang-Baxter equation.

Set-theoretic solutions of the Yang-Baxter equation are investigated extensively in recent years, see for example
\cite{Cedo,Cedo-Jespers-Okninski,Etingof-Schedler-Soloviev,Gateva-Ivanova-Van den Bergh,Lu-Yan-Zhu,Rump1,Rump2,Soloviev}. In particular,   Rump introduced left braces in \cite{Rump2} as a generalization of Jacobson radical rings, and a few years later, Ced$\acute{\rm o}$, Jespers, and Okni$\acute{\rm n}$ski \cite{Cedo-Jespers-Okninski} reformulated
Rump's definition of left braces.  Since then, left braces have become the most used tool in the investigations of  non-degenerate involutive solutions of the Yang-Baxter equation. In order to study non-degenerate solutions that are not necessarily involutive, Guarnieri and Vendramin \cite{Guarnieri-Vendramin} introduced a generalization of left braces, namely, skew left braces.
Let $(S, +)$ and $(S, \cdot)$ be two groups. The triple $(S, +, \cdot)$ is called a {\em skew left brace}  if
\begin{equation}\label{skew brace}x(y+z)=xy-x+xz \mbox{ for all }x,y,z\in S.
\end{equation}
In \cite[Theorem 3.1]{Guarnieri-Vendramin}, Guarnieri and Vendramin have proved that every skew left brace $(S, +, \cdot)$ can give rise to some bijective non-degenerate solution.

To investigate left  non-degenerate solutions of the Yang-Baxter equation, Catino, Colazzo  and Stefanelli \cite{Catino-Colazzo-Stefanelli} introduced left cancellative semibraces. Recall that a semigroup $(S, \cdot)$ is called {\em left cancellative} if for all $a,b,c\in S$, $ab=ac$ implies that $b=c$.
A triple $(S, +, \cdot)$ is called a {\em left cancellative semibrace} if $(S, +)$ is a  left cancellative semigroup, $(S, \cdot)$ is a group and the following axiom holds:
\begin{equation}\label{semibraces}
x(y+z)=xy+x(x^{-1}+z).
\end{equation}
One can show that (\ref{skew brace}) and (\ref{semibraces}) coincide in a skew left brace. Since groups are  left cancellative semigroups, it follows that skew left braces are left cancellative semibraces.    More results about left cancellative semibraces and left non-degenerate solutions of the Yang-Baxter equation have been provided in \cite{Castelli,Catino-Cedo-Stefanelli,Colazzo,Colazzo-Jespers-Van Antwerpen-Verwimp}. In 2019, Jespers and Van Antwerpen \cite{Jespers-Van Antwerpen} introduced and investigated general left semibraces.
A triple $(S, +, \cdot)$ is called a {\em left  semibrace} if $(S, +)$ is a semigroup, $(S, \cdot)$ is a group and the axiom (\ref{semibraces}) holds. Among other things, Jespers and Van Antwerpen have shown that under a natural assumption,   a left semibrace $(S, +, \cdot)$ can provide a solution (see \cite[Theorem 5.1]{Jespers-Van Antwerpen} and Theorem \ref{solution} in this paper). More  results on general left semibraces can be found in \cite{Colazzo-Van Antwerpen,Stefanelli}. Recently, Doikou and Rybo{\l}owicz
\cite{DoRy22x,DoRy22x1} obtained  new multi-parametric, nondegenerate, noninvolutive  set-theoretic solutions of the Yang--Baxter
equation by using the notion of the near left braces introduced there.

In \cite{Catino-Colazzo-Stefanelli}, Catino, Colazzo  and Stefanelli have shown that the additive semigroup $(S,+)$ of a left cancellative left semibrace $(S,+,\cdot)$ is right group (i.e. a direct product of a group and a right zero semigroup). In \cite{Catino-Colazzo-Stefanelli3}, Catino, Colazzo and Stefanelli have pointed out that  only partial results on additive semigroups are known for left semibraces. More specifically, under some mild assumptions, for instance in the finite case, \cite[Theorem 2.8]{Jespers-Van Antwerpen} states that the additive semigroup $(S,+)$ of a left semibrace $(S,+,\cdot)$ is a completely simple semigroup (In fact, by \cite[Theorem 2.3]{Catino-Colazzo-Stefanelli3}, $(S,+)$ is a rectangular group, i.e., it is isomorphic to the direct product of a group and a rectangular band).
On the other hand, Catino, Mazzotta and Stefanelli \cite[Proposition 5]{Catino-Mazzotta-Stefanelli} have shown that the additive semigroup $(S,+)$ of a left semibrace $(S,+,\cdot)$ is a {\em rectangular semigroup}, i.e.
\begin{equation}\label{zhi}
a+x=b+x=a+y \Longrightarrow a+x=b+y \mbox{ for all } a,b,x,y\in S.
\end{equation}
However, the structure of additive semigroups of general left semibraces still remains mysterious.

In this paper, we  study left semibraces, near left semibraces and the Yang-Baxter equation. In Section 2, we prove that the additive semigroups of   left semibraces are always  rectangular groups. Section 3 gives some characterizations of strong left semibraces. In Section 4, we introduce the notion of  near left semibraces and prove that a strong left semibrace can induce a near left semibrace, and vice versa. As a consequence,  we have that near left semibraces can provide set-theoretical solutions for the Yang-Baxter equation.  Section 5 provide a structure theorem for general left semibraces by the generalized matched products of right zero left semibraces and right cancellative left semibraces. In the final section, we introduce a new map associated to a  near left semibrace and give a sufficient and necessary condition under which such a map forms a set-theoretic  solution of the Yang-Baxter equation.
Our result improve and enrich some results  obtained in \cite{Catino-Colazzo-Stefanelli4,Catino-Mazzotta-Stefanelli,Jespers-Van Antwerpen} for left semibraces.

\section{The additive semigroup of a left semibrace is  a rectangular group}
In this section, we shall prove that the additive semigroup of a left semibrace forms  a rectangular group. We first restate the notion of left semibraces. To avoid parentheses,  throughout this paper we will assume that the multiplication has higher precedence than the addition.  As usual, we denote  the set of idempotents in a semigroup $(S,+)$ by $E(S,+)$.
\begin{defn}[\cite{Jespers-Van Antwerpen}] A triple $(S, +, \cdot)$ is called a {\em left semibrace} if $(S, +)$  is a semigroup, $(S, \cdot)$ is a group and
\begin{equation}\label{dengshi}x(y+z)=xy+x(x^{-1}+z) \mbox{ for all } x, y, z\in S,
\end{equation} where $x^{-1}$ is the inverse of $x$ in the group $(S, \cdot)$. {\em If this is the case, throughout this note, we always denote the identity in the group $(S, \cdot)$ by $e$.}
\end{defn}

The following lemma gives some basic properties of left semibraces, where items (2) and (3) have been already proved in \cite{Jespers-Van Antwerpen}.
\begin{lemma}\label{key1}Let $(S,+,\cdot)$ be a left semibrace and $x,y,z\in S$.
\begin{itemize}

\item[(1)] $x(x^{-1}+y)=e+x(x^{-1}+y)$,~$x+x(x^{-1}+y)=x(e+y)$ and $x^{-1}+x^{-1}(x+y)=x^{-1}(e+y)$.
\item[(2)] $x+e+y=x+y$ As a consequence, $e+e+e=e+e$.
\item[(3)] $e+e=e$. That is, $e\in E(S,+)$.
\item[(4)]  $x+y=x+u+y$ for all $u\in E(S,+)$.
\item[(5)] $(x+y)^{-1}+e=(x+y)^{-1} y (y^{-1}+e)$ and $(x+y)((x+y)^{-1}+e)= y (y^{-1}+e)$.
\item[(6)] $e+(e+x)(e+y)=(e+x)(e+y)$.
\item[(7)] $xy+z=x(y+x^{-1}(x+z))$.
\end{itemize}
\end{lemma}
\begin{proof}
(1) By (\ref{dengshi}), we have $$x(x^{-1}+y)=xx^{-1}+x(x^{-1}+y)=e+x(x^{-1}+y)$$ and
$$x+x(x^{-1}+y)=xe+x(x^{-1}+y)=x(e+y).$$ Substituting $x$ by $x^{-1}$, we obtain $x^{-1}+x^{-1}(x+y)=x^{-1}(e+y)$.

(2)  By (\ref{dengshi}), we have $$x+y=e(x+y)=ex+e(e^{-1}+y)=xe+e(e+y)=x+e+y.$$

(3) Let $x=e+e$. Then $x+e=e+x=x=x+x$ by (2). This implies that
$$x^2=x(e+e)\overset{\rm (\ref{dengshi})}{=}xe+x(x^{-1}+e)=x+x(x^{-1}+x^{-1}x)$$$$=x+x(x^{-1}+x^{-1}(x+e))\overset{\rm (1)}{=}x+x x^{-1}(e+e)=x+x x^{-1}x=x+x=x.$$
This gives that $e+e=x=e$.

(4) Let $u\in E(S,+)$. Then $u=u+u$. By (1), $$u^{-1}(e+u+y)=u^{-1}+u^{-1}(u+u+y)=u^{-1}+u^{-1}(u+y)=u^{-1}(e+y).$$
This implies that $e+u+y=e+y$, and so $$x+u+y\overset{\rm (2)}{=}x+e+u+y=x+e+y\overset{\rm (2)}{=}x+y.$$

(5) Substituting $x$ by $x+y$ and $y$ by $y(y^{-1}+e)$ in the third equality in (1), respectively, $$(x+y)^{-1}+(x+y)^{-1}(x+y+y(y^{-1}+e))=(x+y)^{-1}(e+y(y^{-1}+e)).$$ This gives that
$$(x+y)^{-1}+(x+y)^{-1}(x+y(e+e))=(x+y)^{-1} y(y^{-1}+e)$$ by item (1).  This together with item (3) implies that $$(x+y)^{-1}+e=(x+y)^{-1} y(y^{-1}+e),$$ and so
$(x+y)((x+y)^{-1}+e)= y (y^{-1}+e)$.

(6) By item (3) in the present lemma, we have $e+e=e$, and so $$e+(e+x)(e+y)=e+(e+x)e+(e+x)((e+x)^{-1}+y)$$$$=e+e+x+(e+x)((e+x)^{-1}+y)=e+x+(e+x)((e+x)^{-1}+y)$$$$=(e+x)e+(e+x)((e+x)^{-1}+y)=(e+x)(e+y).$$

(7) By (\ref{dengshi}) and items (1) and (2) of the present lemma, we have  $$x(y+x^{-1}(x+z))=xy+x(x^{-1}+x^{-1}(x+z))$$$$=xy+xx^{-1}(e+z)=xy+e+z=xy+z,$$ which gives item (7).
\end{proof}
To give the first main result of this note, we first recall the notion of rectangular groups. Let $(G,+)$ be a group with identity $0$ and $L$ and $R$ be a left zero semigroup and a right zero semigroup, respectively. Then $M=I\times G\times \Lambda$ forms a semigroup with respect to the following addition:
$$(i,g,\lambda)+(j,h,\mu)=(i,g+h, \mu)\mbox{  for all }(i,g,\lambda),(j,h,\mu)\in I\times G\times \Lambda.$$
Observe that $$E(M,+)=\{(i,0,\lambda)\mid i\in I, \lambda\in \Lambda\}.$$ A semigroup which is isomorphic to some semigroup constructed above is called a {\em rectangular group}.  A semigroup $(S,+)$ is called {\em regular} if for all $a\in S$, there is $a'\in S$ such that $a+a'+a=a$. A regular semigroup $(S,+)$ is called {\em orthodox} if $E(S,+)$ is a subsemigroup of $(S,+)$.  Rectangular groups are orthodox semigroups.  A semigroup $(S,+)$ is called  {\em left cancellative} (respectively, {\em right cancellative}) if for all $x,y,z\in S$, the equation $z+x=z+y$ (respectively, $x+z=y+z$) implies that $x=y$.
 \begin{lemma}\label{rectangular}
Let $(S,+)$ be a rectangular group, $x,y\in S$ and $u\in E(S,+)$.

(1) If $x+y=u$, then $u+x=x$.

(2) There exists $v\in E(S,+)$ such that $v+x=x+v=x$.

(3) If $x+u=u$ or $u+x=u$, then $x\in E(S,+)$.

(4) For all $a\in S$, there is a unique $\pi_u(a)\in S$ such that $$a+\pi_u(a)+a=a, \pi_u(a)+a+\pi_u(a)=\pi_u(a) \mbox{ and }\pi_u(a)+u=\pi_u(a)=u+\pi_u(a).$$
We call this $\pi_u(a)$ the negative of $a$ with respect  to $u$.
Obviously, $\pi_u(u)=u$. If $a,b\in S$ and $x\in E(S,+)$, we have $$\pi_u(x), \pi_u(a)+a, a+\pi_u(a)\in E(S,+) \mbox{ and } \pi_u(a+b)=\pi_u(b)+\pi_u(a).$$

(5) $(u+S,+)$ and $(S+u,+)$ are left and right cancellative orthodox semigroups, respectively, where $u+S=\{u+s\mid s\in S\}$ and $S+u=\{s+u\mid s\in S\}$. Moreover, $E(u+S)=\{e+x\mid x\in E(S,+)\}=u+E(S,+)$ and $E(S+u)=\{x+u\mid x\in E(S,+)\}=E(S,+)+u$.
\end{lemma}
\begin{proof} Assume that $S=I\times G\times \Lambda$ and $x=(i,g,\lambda), y=(j,h,\mu)$ and $u=(k,0, \nu)$.

(1) If $x+y=u$, then $i=k$ and so $u+x=x$.

(2) Obviously, $v=(i,0,\lambda)\in E(S,+)$ and $v+x=x+v=x$.

(3) If $x+u=u$ or $u+x=u$, then $g=0$, and so $x\in E(S,+)$.

(4) Let $a=(i,g,\lambda)$ and $b=(j,h,\mu)$. Take $\pi_u(a)=(k,-g,\nu)$ and $\pi_u(b)=(k,-h,\nu)$. Then it is routine to check that the statements in item (4) hold.
If $a'=(s,q, \omega)\in S$, $a+a'+a=a$ and $a'+u=u+a'=a'$. Then $g+q+g=g$  and $s=k, \omega=\nu$, which implies that $a'=\pi_u(a)$.

(5) This can be obtained by simple calculations.
\end{proof}

The following result improves \cite[Proposition 5]{Catino-Mazzotta-Stefanelli}.
\begin{theorem}\label{main1}
Let $(S,+,\cdot)$ be a left semibrace. Then $(S,+)$ forms a rectangular group.
\end{theorem}
\begin{proof}
Let $x,y,z\in S$. Then by (\ref{dengshi}) and (1) and (2) in Lemma \ref{key1}, we have  $$x(y+x^{-1}(x+z))=xy+x(x^{-1}+x^{-1}(x+z))=xy+xx^{-1}(e+z)=xy+e+z=xy+z.$$
Take $y=x^{-1}+x^{-1}$ in the above equality. Then we obtain
\begin{equation}\label{key2}
x(x^{-1}+x^{-1})+z=x(x^{-1}+x^{-1}+x^{-1}(x+z))=x(x^{-1}+ x^{-1}(e+z))
\end{equation} by Lemma \ref{key1} (1). Take $x=e+z$ in (\ref{key2}). By Lemma \ref{key1} (5), we
have $$(e+z)((e+z)^{-1}+(e+z)^{-1})+z$$$$=(e+z)((e+z)^{-1}+ (e+z)^{-1}(e+z))=(e+z)((e+z)^{-1}+ e)=z(z^{-1}+e).$$
Adding $z$ on the left side of the above equality, we have $$z+(e+z)((e+z)^{-1}+(e+z)^{-1})+z=z+z(z^{-1}+e)=z(e+e)=ze=z$$ by  (1) and (3) in Lemma \ref{key1}.
Thus $(S,+)$ forms a regular semigroup. By \cite[Proposition 5]{Catino-Mazzotta-Stefanelli}, $(S,+)$ is a rectangular  semigroup (see (\ref{zhi}) for the definition).
In view of \cite[Exercises III.5.15 (iii)]{Petrich-Reilly},  $(S,+)$ forms a rectangular group.
\end{proof}

\section{Strong left semibraces}
In the theory of left semibraces, the following maps are important. Let $(S,+,\cdot)$ be a left semibrace and $a\in S$. Define
$$\overline{\lambda}_{a} : S \longrightarrow S: b\mapsto a(a^{-1}+b),\,\,\,\,\,\, \overline{\rho}_a :  S \longrightarrow S : b \mapsto   (b^{-1}+a)^{-1}a,$$
$$\hat{r}: (S, \cdot) \longrightarrow \textup{Map}(S,S),\, a\mapsto \overline{\lambda}_a\,\,\,\, \overline{\rho}: (S, \cdot) \longrightarrow \textup{Map}(S,S),\, a\mapsto \overline{\rho}_a.$$
Observe that $e+S$ and $S+e$ are subsemigroups of $(S,+)$ by Lemma \ref{key1} (2). By \cite[Lemmas 2.12 and 2.13]{Jespers-Van Antwerpen}, we have the following result.
\begin{lemma}[\cite{Jespers-Van Antwerpen}]\label{lambdaidempot}
Let $(S,+,\cdot)$  be a left semibrace. For any $a \in S$, $\overline{\lambda}_a \in \textup{End}(S,+)$. Moreover, $\overline{\lambda}$ is a semigroup
	homomorphism. Furthermore, for any $a \in S$, we have   $$\overline{\lambda}_a(S)\subseteq e+S,\,\,\, \overline{\lambda}_a(E(S,+))
	\subseteq E(e+S),\,\,\,\overline{ \rho}_a(S)\subseteq S+e.$$
\end{lemma}
It is well known that $\overline{\rho}$ is a semigroup anti-homomorphism in the case of skew left braces. However, $\overline{\rho}$ may not be a semigroup anti-homomorphism in general left semibraces, see \cite{Jespers-Van Antwerpen}. For this reason, we propose the following notion for convenience.
\begin{defn}
A left semibrace $(S,+,\cdot)$ is called {\em strong} if $\overline{\rho}$ is a semigroup anti-homomorphism.
\end{defn}
We observe that strong left semibraces have been investigated extensively in \cite{Jespers-Van Antwerpen}.
In this section, we shall give some characterizations for  strong left semibraces.
Our result improves  \cite[Proposition 2.14 and Theorem 1]{Jespers-Van Antwerpen} and \cite[Theorem 3]{Catino-Colazzo-Stefanelli4}.
\begin{theorem}\label{solution} For a left semibrace $(S,+,\cdot)$, the following statements are equivalent:
\begin{itemize}
\item[(1)] $\rho$ is strong.
\item[(2)] $c+a(e+b)=c+ab$, for all $a, b, c\in S$.
\item[(3)] $e+a+e=e+au+e$, for all $a\in S$ and $u\in E(S,+)$.
\item[(4)] The map $\overline{r}_S:S\times S\rightarrow S\times S, (a,b)\mapsto (\overline{\lambda}_a(b),\overline{\rho}_b(a))$ is a  set-theoretic  solution  of the Yang-Baxter equation.
\item[(5)] $a+b(b^{-1}+c)(e+(b^{-1}+c)^{-1}c)=a+b(e+c)$,  for all $a, b, c\in S$.
\end{itemize}
\end{theorem}
\begin{proof}

{\em (1) is equivalent to (2)}. See \cite[Proposition 2.14]{Jespers-Van Antwerpen}.

{\em (2) implies (3)}. Let $a\in S$ and $u\in E(S,+)$. Then by item (2), we have $$e+au+e=e+a(e+u)+e=e+ae+a(a^{-1}+u)+e=e+a+\lambda_a(u)+e.$$
By Lemma \ref{lambdaidempot}, $\lambda_a(u)\in E(S,+)$. So by Lemma \ref{key1} (4), we have $$e+au+e=e+a+\lambda_a(u)+e=e+a+e.$$

{\em (3) implies (2)}. Let $a,b,c\in S$.
Since $(S,+)$ is a rectangular group by Theorem \ref{main1}, it follows that $b=v+b$  for some $v\in E(S,+)$ by Lemma \ref{rectangular}. This implies that
$$c+ab=c+a(v+b)=c+av+a(a^{-1}+b)=c+e+av+e+a(a^{-1}+b)$$$$=c+e+a+e+a(a^{-1}+b)=c+a+a(a^{-1}+b)=c+a(e+b)$$ by using (\ref{dengshi}) and Lemma \ref{key1} (2), (1) and item (3) in the present lemma.

{\em (4) is equivalent to (5)}. See \cite[Theorem 3]{Catino-Colazzo-Stefanelli4}.

{\em (2) implies (5)}. If (2) holds, then both sides of the identity in (5) are equal to $a+bc$, and so  (5) holds.

{\em (5) implies (2)}. Assume that
\begin{equation}\label{wang1}
a+b(b^{-1}+c)(e+(b^{-1}+c)^{-1}c)=a+b(e+c)
\end{equation}  for all $a, b, c\in S$. Let $a=b=e$ and $c=y$ in (\ref{wang1}) and denote $t=e+y$. Then
$e+t(e+t^{-1}y)=e+y$.  By Lemma \ref{key1} (7) and (\ref{dengshi}),
$$t(e+t^{-1}y)+z=t(e+t^{-1}y+t^{-1}(t+z))=t(e+t^{-1}(y+z)).$$ This implies that
\begin{equation}\label{wang2}
\begin{array}{cc}
e+y+z=e+t(e+t^{-1}(y+z))
=e+(e+y)(e+t^{-1}(y+z))
\\[2mm]=(e+y)(e+t^{-1}(y+z))=(e+y)(e+(e+y)^{-1}(y+z))
\end{array}
\end{equation}
by Lemma \ref{key1} (6), and so
\begin{equation}\label{wang3}
 (e+y)^{-1}(e+y+z)=e+(e+y)^{-1}(y+z).
\end{equation}
By Lemma \ref{key1} (1) and (2), we obtain that
\begin{equation}\label{wang4}
\begin{array}{cc}
(e+y)^{-1}(e+z)=(e+y)^{-1}+ (e+y)^{-1}(e+y+z)\\[2mm]
=(e+y)^{-1}+e+(e+y)^{-1}(y+z)=(e+y)^{-1}+(e+y)^{-1}(y+z).
\end{array}
\end{equation}
Let $a=x, b=(zy)^{-1}$ and $c=u$ in (\ref{wang1}). Then
$$x+y^{-1}z^{-1}(zy+u)(e+(zy+u)^{-1}u)=x+y^{-1}z^{-1}(e+u).$$
By Lemma \ref{key1} (7), we have $zy+u=z(y+z^{-1}(z+u))$. So
\begin{equation}\label{wang5}
x+y^{-1}(y+z^{-1}(z+u))(e+(y+z^{-1}(z+u))^{-1}z^{-1}u)=x+y^{-1}z^{-1}(e+u).
\end{equation}
Substituting $z$ and $u$ by $z^{-1}$ and $z^{-1}(z+u)$ in (\ref{wang5}), respectively, and denoting $A=y+z(z^{-1}+z^{-1}(z+u))$, we have
\begin{equation}\label{wang6}x+y^{-1}A(e+A^{-1}(z+u))=x+y^{-1}z(e+z^{-1}(z+u)).
\end{equation}
By Lemma \ref{key1} (1), (2), we obtain $$A=y+z(z^{-1}+z^{-1}(z+u))=y+zz^{-1}(e+u)=y+e+u=y+u$$ and $e+z^{-1}(z+u)=z^{-1}(z+u)$. Thus (\ref{wang6}) becomes
\begin{equation}\label{wang7}x+y^{-1}(y+u)(e+(y+u)^{-1}(z+u))=x+y^{-1}(z+u)).\end{equation}
Substituting $z$ and $u$ by $(e+z)^{-1}$ and $(e+z)^{-1}(z+u)$ in (\ref{wang5}), respectively, and denoting $$B=y+(e+z)((e+z)^{-1}+(e+z)^{-1}(z+u)),$$$$ C=(e+z)(e+(e+z)^{-1}(z+u)),$$ we have $x+y^{-1}B(e+B^{-1}(z+u))=x+y^{-1}C$. By (\ref{wang4}), (\ref{wang3}) and Lemma \ref{key1} (2), we obtain that $$B=y+(e+z)(e+z)^{-1}(e+u)=y+e+u=y+u$$
and $$C=(e+z)(e+z)^{-1}(e+z+u)=e+z+u.$$ This implies that $$x+y^{-1}(y+u)(e+(y+u)^{-1}(z+u))=x+y^{-1}(e+z+u).$$  This together (\ref{wang7}) gives that
$$x+y^{-1}(e+z+u)=x+y^{-1}(z+u)).$$ Substituting $y$ by $y^{-1}$ and taking $u$ such that $z+u=z$ ( by Theorem \ref{main1} and Lemma \ref{rectangular} (2)), we have
$x+y(e+z)=x+yz$. Thus (2) holds.
\end{proof}
\section{Near left semibraces}
In this section, we introduce the notion of {\em near left semibraces} and prove that a strong left semibrace can induce a near left semibrace, and vice versa. As a consequence, near left semibraces can provide set-theoretical solutions for the Yang-Baxter equation. Some new  characterizations  of strong left semibraces are also obtained.
\begin{defn}Let $(S, +)$ be a semigroup, $(S,\cdot)$ be a group with identity $e$ and $\mu: S\longrightarrow S, x\mapsto \mu(x)$ be a map. Then $(S,+,\cdot, \mu)$ is called a {\em near left semibrace} if
\begin{equation}\label{near}
x(y+z)=xy+\mu(x)+xz  \mbox{ for all } x,y,z\in S.
\end{equation}
\end{defn}
\begin{remark}\label{zhun}
Recall from Doikou  and Rybo{\l}owicz \cite{DoRy22x,DoRy22x1} that a triple $(G,+, \cdot)$ is called a {\em near left brace} if $(G,+)$ is a group with identity $0$, $(G,\cdot)$ is a group and $x(y+z)=xy-x\cdot 0+xz$ for all $x,y,z\in G$.
Let $(S,+,\cdot)$ be a near left semibrace. Let $x\in S$. If $(S,+)$ is also a group with identity $0$, then we have $x\cdot0=x\cdot(0+0)=x\cdot0+\mu(x)+x\cdot 0$. This implies that $\mu(x)=-x\cdot0$. In particular, $\mu(e)=-e\cdot 0=-0=0$. Thus $(S, +, \cdot)$ forms a near left brace.  This shows that near left semibraces are indeed generalizations of near left braces.
\end{remark}

\begin{lemma}\label{near1}Let $(S,+,\cdot, \mu)$ be a near left semibrace.
\begin{itemize}
\item[(1)] $x+\mu(e)+y=x+y$ for all $x,y\in S$.
\item[(2)] $x+\mu(y\mu(e)^{-1})+y+z=x+z$ and $x+z=x+\mu(y)+y\mu(e) +z$ for all $x,y,z\in S$.
\item[(3)] If $e+\mu(x)=\mu(x)$ (respectively, $\mu(x)+e=\mu(x)$, or $\mu(x)+x+\mu(x)=\mu(x)$, or $x+\mu(x)+x=x$) for all $x\in S$, then $x+e+y=x+y$ for all $x,y\in S$.
\end{itemize}
\end{lemma}

\begin{proof}
 (1)  Let $x,y\in S$. By (\ref{near}), we have
\begin{equation} \label{wang90}
x+y=e(x+y)=ex+\mu(e)+ey=x+\mu(e)+y \mbox{ for all }x,y\in S.
\end{equation}

(2)   Let $x,y,z\in S$. By (\ref{near}) and item (1), we have
$$x^{-1}(y+z)=x^{-1}y+\mu(x^{-1})+x^{-1}z=x^{-1}y+\mu(e)+\mu(x^{-1})+x^{-1}z=x^{-1}y+x^{-1}(x\mu(e)+z).$$
This implies that $$y+z=x(x^{-1}y+x^{-1}(x\mu(e)+z))=y+\mu(x)+x\mu(e)+z.$$
Substituting $x$ by $x\mu(e)^{-1}$, we obtain that $$y+z=y+\mu(x\mu(e)^{-1})+x +z.$$  Substituting $x$ by $x \mu(e)$, we have $y+z=y+\mu(x)+x\mu(e) +z$. This  gives item (2).

(3) Let $x,y,z\in S$. If $e+\mu(x)=\mu(x)$ for all $x\in S$, then by item (1), we have $$x+y=x+\mu(e)+y=x+e+\mu(e) +y=x+e+y.$$
If $\mu(x)+e=\mu(x)$ for all $x\in S$, then by  item (1), we have $$x+y=x+\mu(e)+y=x+\mu(x)+e +y=x+e+y.$$
If $\mu(x)+x+\mu(x)=\mu(x)$ for all $x\in S$, then by item (1), we have $$x+y=x+\mu(e)+y=x+\mu(e)+e+\mu(e) +y=x+e+y.$$
If $x+\mu(x)+x=x$ for all $x\in S$, then
$\mu(e)+\mu(\mu(e))+\mu(e)=\mu(e).$
This implies that
\begin{equation*}
x+y=x+\mu(e)+y=x+\mu(e)+\mu(\mu(e))+\mu(e)+y=x+\mu(\mu(e))+y,
\end{equation*}
\begin{equation*}\label{wang91} \mu(e)(x+y)=\mu(e)x+\mu(\mu(e))+\mu(e)y=\mu(e)x+\mu(e)y
\end{equation*}
for all $x,y\in S$ by item (1) and (\ref{near}).
This gives that $$\mu(e)(x+e+y)=\mu(e)x+\mu(e)e+\mu(e)y$$$$=\mu(e)x+\mu(e)+\mu(e)y=\mu(e)x+\mu(e)y=\mu(e)(x+y),$$
by item (1), and so  $x+e+y=x+y$.
\end{proof}

\begin{theorem}\label{main2}
Let $(S,+,\cdot, \mu)$ be a near left semibrace, where the identity of the group $(S,\cdot)$ is $e$ and $x^{-1}$ is the inverse of $x$ in $(S,\cdot)$ for all $x\in S$. Define a new binary operation $\circ$ on $S$ as follows: For all $x,y\in S$, let $x\circ y=x\mu(e)^{-1}y$.
Then $(S,+,\circ)$ forms a strong left semibrace, and so $(S,+)$ is a rectangular group and $\mu(e)\in E(S,+)$. Moreover, if we denote $\overline{x}=\mu(e)x^{-1}\mu(e)$ for all $x\in S$, then
$$\overline{r}_S: S\times S\to S\times S,\, (x,y)\mapsto (x\circ(\overline{x}+y),\overline{\overline{x}+y}\circ y)$$
is a set-theoretical solution of the Yang-Baxter equation.
\end{theorem}
\begin{proof}It is well known that $(S,\circ)$ forms a group with identity $\mu(e)$, and the inverse of $x$ in this group is $\overline{x}=\mu(e)x^{-1}\mu(e)$ for all $x\in S$. Let $x,y,z\in S$. Then $$x\circ y+x\circ(\overline{x}+z)=x\mu(e)^{-1} y+x\mu(e)^{-1}(\mu(e)x^{-1}\mu(e)+z)
$$$$\overset{\rm (\ref{near})}{=}x\mu(e)^{-1} y+\mu(e)+\mu(x\mu(e)^{-1})+x\mu(e)^{-1} z$$$$\overset{\rm Lemma\, \ref{near1}\, (1)}{=}x\mu(e)^{-1} y+ \mu(x\mu(e)^{-1})+x\mu(e)^{-1} z=x\mu(e)^{-1}(y+z)=x\circ (y+z).$$
Thus $(S,+,\circ)$ forms a left semibrace. By Theorem \ref{main1}, $(S,+)$ is a rectangular group. By Lemma \ref{key1} (3), we have $\mu(e)\in E(S,+)$.
By Lemma \ref{rectangular} (4), we obtain that $y+\pi_{\mu(e)}(y)\in E(S,+)$. In view of (\ref{near}), Lemma \ref{key1} (4) and Lemma \ref{near1} (2), we have
$$x+y\circ(\mu(e)+z)=x+y\mu(e)^{-1}(\mu(e)+z)=x+y+\mu(y\mu(e)^{-1})+y\mu(e)^{-1}z$$$$=x+y+\mu(y\mu(e)^{-1})+ (y+\pi_{\mu(e)}(y)) +y\mu(e)^{-1}z$$$$=x+y+(\mu(y\mu(e)^{-1})+ y)+\pi_{\mu(e)}(y) +y\mu(e)^{-1}z
$$$$=x+(y+\pi_{\mu(e)}(y)) +y\mu(e)^{-1}z=x+y\mu(e)^{-1}z=x+y\circ z$$
By Theorem \ref{solution}, it follows that $(S,+,\circ)$ is strong and $\overline{r}_S$ is a set-theoretical solution of the Yang-Baxter equation.
\end{proof}

\begin{lemma}\label{near2}
Let $(S,+,\cdot, \mu)$ be a near left semibrace. Then the following statements are equivalent:
\begin{itemize}
\item[(1)]$x+e+y=x+y$ for all $x,y\in S$.
\item[(2)] $(S,+,\cdot)$ is a strong left semibrace.
\item[(3)] $e+e=e$.
\end{itemize}
\end{lemma}

\begin{proof}{\em (1) implies (2).} Let $x,y,z\in S$. By  (\ref{near}) and item (1),  we have $x^{-1}(xy+z)=y+\mu(x^{-1})+x^{-1}z$ and   $$xy+z=x(y+\mu(x^{-1})+x^{-1}z)=x(y+x^{-1}x+\mu(x^{-1})+x^{-1}z)=x(y+x^{-1}(x+z)),$$
$$x(e+z)=x+\mu(x)+xz=x+e+\mu(x)+xz=x+x(x^{-1}+z).$$ Substituting $z$ by $x(x^{-1}+z)$ in the first equality above, we obtain that
$$xy+x(x^{-1}+z)=x(y+x^{-1}(x+x(x^{-1}+z)))=x(y+x^{-1}(x(e+z)))=x(y+e+z)=x(y+z),$$  which implies that $(S,+,\cdot)$ is a left semibrace.
So $e+e=e$ by Lemma \ref{key1} (3). This gives that $$x=xe=x(e+e)=xe+\mu(x)+xe=x+\mu(x)+x$$ and $x+\mu(x)\in E(S,+)$. Thus, by Lemma \ref{key1} (4) and the given equality,
$$z+x(e+y)=z+xe+\mu(x)+xy=z+x+\mu(x)+xy=z+xy.$$ Therefore $(S,+,\cdot)$ is also strong by Theorem \ref{main1} (2).

{\em (2) implies (3).} This follows from Lemma \ref{key1} (3).

{\em (3) implies (1).}  Let $e+e=e$ and $x\in S$. Then $$x=xe=x(e+e)=xe+\mu(x)+xe=x+\mu(x)+x.$$
By Lemma \ref{near1} (3), we obtain item (1).
\end{proof}

\begin{theorem}Let $(S, +)$ be a semigroup and $(S,\cdot)$ be a group with identity $e$. Then  $(S,+,\cdot)$ is a strong left semibrace if and only if there exists a map $\mu: S\longrightarrow S, x\mapsto \mu(x)$ such that $e+\mu(x)=\mu(x)$ (respectively, $\mu(x)+e=\mu(x)$, or $x+\mu(x)+x=x$, or  $\mu(x)+x+\mu(x)=\mu(x)$) such that $(S,+,\cdot,\mu)$ forms near left semibrace.
\end{theorem}
\begin{proof}
Let $(S,+,\cdot)$ be a strong left semibrace. Then by Theorem \ref{main1} and Lemma \ref{key1}, $(S,+)$ forms a rectangular group and $e\in E(S,+)$. Take $\mu=\pi_e$, which is given in Lemma \ref{rectangular} (4). Then $$e+\mu(x)=\mu(x),\,\mu(x)+e=\mu(x),\,x+\mu(x)+x=x,\,\mu(x)+x+\mu(x)=\mu(x)$$ by   Lemma \ref{rectangular} (4). Obviously, $\mu(x)+x\in E(S,+)$. By Lemma \ref{key1} and Theorem \ref{solution} (2), we obtain that  $$xy+x(x^{-1}+z)=xy+\mu(x)+x+x(x^{-1}+z)=xy+\mu(x)+x(e+z)=xy+\mu(x)+xz.
$$
Thus the necessity is true. The sufficiency follows from Lemmas \ref{near1} and \ref{near2}.
\end{proof}

\section{Generalized matched products of left semibraces}
In this section, we give a structure theorem for general left semibraces by used generalized matched products of left semibraces, and   improves \cite[Theorem 3.2]{Jespers-Van Antwerpen} which gives a structure theorem for strong left semibraces. We begin by the following lemma.
\begin{lemma}\label{subbrace}
Let $(S,+,\cdot)$  be a left semibrace.
\begin{itemize}
\item[(1)] $(S+e, +, \cdot)$ is  a  right cancellative left semibrace (i.e. (S+e,+) is right cancellative).
\item[(2)]  $(e+S, +, \cdot)$ is  a  left cancellative left semibrace  (i.e. (S+e,+) is left cancellative).
\item[(3)]  $(e+S+e, +, \cdot)$ is  a  skew left barce.
\item[(4)] $(E(e+S), +, \cdot)$ is  a right zero left semibrace  (i.e. (e+S+e,+) is right zero).
\end{itemize}
\end{lemma}

\begin{proof}By Theorem \ref{main1}, $(S,+)$ is a rectangular group.
Item (1) follows from \cite[Lemma 2.6 (ii)]{Jespers-Van Antwerpen} and Lemma \ref{rectangular}.
 \cite[Lemma 2.6 (iii)]{Jespers-Van Antwerpen} gives that $(e+S, \cdot)$ is a semigroup.
Now, let $x\in S$ and $a=e+x$. Then $$E(S,+)\ni e=a^{-1}a=a^{-1}(e+x)=a^{-1}e+a^{-1}((a^{-1})^{-1}+x)=a^{-1}+a^{-1}(a+x).$$
This implies that $a^{-1}=e+a^{-1}\in e+S$ by Lemma \ref{rectangular} (1). Thus (2) holds by Lemma \ref{rectangular}. By items (1), (2) above and \cite[Proposition 2.18]{Jespers-Van Antwerpen}, we have item (3). To prove (4),  we let $x,y\in E(S,+)$. Then $x+x=x$ and $y+y=y$. By the third equality in Lemma \ref{key1} (1), we have
$$(e+x)^{-1}+(e+x)^{-1}(e+x+x)=(e+x)^{-1}(e+x).$$ This implies that $(e+x)^{-1}+e=e$,  and so $(e+x)^{-1}\in E(S,+)$ by Lemma \ref{rectangular}.
By Lemma \ref{key1}, we have $e+(e+x)^{-1} +y=e+y$.  Moreover, we also have $(e+x)^{-1}=e+(e+x)^{-1}$ by item (2).  So $e+y=e+(e+x)^{-1} +y=(e+x)^{-1} +y$.
Thus $$(e+x)(e+y)+(e+x)(e+y)=(e+x)(e+y)+(e+x)((e+x)^{-1} +y)$$$$=(e+x)(e+y+y)=(e+x)(e+y),$$ which gives that $(e+x)(e+y)\in E(e+S)$. Therefore (4) holds by Lemma \ref{rectangular}.
\end{proof}

By \cite[Theorem 3.1.1]{Colazzo} and its proof, we have the following result.
\begin{lemma}[\cite{Colazzo}]\label{colazzo}
Let $(K,+,\cdot)$ and $(R,+,\cdot)$ be two  left semibraces. Let $\delta: (K, \cdot) \longrightarrow \textup{Sym}(R, +)$ be a right action on the set $R$ and $\sigma: (R,\cdot) \longrightarrow \textup{Sym}(K,+)$ be a left action on the set $K$.
Assume that,   for any $x,y \in R$ and $a,b,c \in K$,
	\begin{enumerate}
		\item $\sigma_x(a  b) = \sigma_x(a)  \,\, \sigma_{ \delta_a(x)}(b),$
		\item $\sigma_x(e_K) = e_K,$
		\item $\delta_a(x   y) = \delta_{\sigma_{y}(a)}(x)\,\,  \delta_a(y),$
		\item $\delta_a(e_R) = e_R,$
		\item $ ({ \delta_a ((x+y)^{-1}) } )^{-1}= ({ \delta_a (x^{-1}) } )^{-1} + ({ \delta_a (y^{-1}) } )^{-1}$,
        \item $a\cdot \sigma_{(\delta_a(x^{-1}))^{-1}}(b+c)=a\cdot \sigma_{(\delta_a(x^{-1}))^{-1}}(b) +a\cdot  \sigma_{(\delta_a(x^{-1}))^{-1}}(\sigma_{(\delta_a(x^{-1}))}(a^{-1})+c)$.
	\end{enumerate}
Then the following operations define a left semibrace structure on $K\times R$:
$$ (a,x)+(b,y) = (a+b, x+y), \,\, ( a,x)  (b,y) = \left( a \cdot
\sigma_{ ({ \delta_{a}(x^{-1})})^{-1}}(b),\,\,\,\, x \cdot ({
\delta_{({ \sigma_{ x^{-1}}(a)})^{-1}}(y^{-1})})^{-1} \right) .
$$
We shall say this left semibrace {\bf the  generalized  matched product} of the left semibrace $K$ and $R$ by $\delta$ and $\sigma$, denoted $K \bowtie^\delta_\sigma R$.
\end{lemma}

Let $(S,+,\cdot)$ be  a left semibrace with multiplication identity $e$ and $ K=S+e, R=E(e+S)$.
Then
$(K,+,\cdot)$ is a  right cancellative left semibrace and  $(R, +, \cdot)$  is  a right zero left semibrace by Lemma \ref{subbrace}. Let
$$\delta: (K, \cdot) \longrightarrow \textup{Sym}(R, +), a\mapsto \delta_a,\,\,\, \delta_a: (R,+)\rightarrow (R,+), x\mapsto (a+x^{-1})^{-1}a;$$
$$\sigma: (R, \cdot) \longrightarrow \textup{Sym}(K, +), x\mapsto \sigma_x,\,\,\, \sigma_x: (K,+)\rightarrow (K,+), a\mapsto x(a+x^{-1}).$$
Let $x\in R$ and $a\in K$. Then $x^{-1}\in R$ and $a^{-1}\in K$. By Lemma \ref{lambdaidempot},  we have $a^{-1}(a+x^{-1})=\lambda_{a^{-1}}(x^{-1})\in R$ and
$(a+x^{-1})^{-1}x^{-1}=\rho_{x^{-1}}(a^{-1})\in K$. This implies that
$\delta_a(x)=(a+x^{-1})^{-1}a=(a^{-1}(a+x^{-1}))^{-1}\in R$ and $$\sigma_x(a)=x(a+x^{-1})=((a+x^{-1})^{-1}x^{-1})^{-1}\in K$$ by Lemma \ref{subbrace}.
Thus the above $\delta$ and $\sigma$ are well-defined.
Now we are in a position to state our second main result of this note, which improves \cite[Theorem 3.2]{Jespers-Van Antwerpen}.
\begin{theorem}\label{characttheorem}Let $(S,+,\cdot)$ be  a left semibrace with multiplication identity $e$. Then
$$(S, +, \cdot) \cong (K,+, \cdot) \bowtie^\delta_\sigma (R, +, \cdot).$$
\end{theorem}
\begin{proof}
Let $x,y\in R$ and $a,b\in K$. Then $x^{-1},y^{-1} \in R, a^{-1}, b^{-1}\in K $ and  $$e+x=x, e+y=y, e+x^{-1}=x^{-1}, e+y^{-1}=y^{-1},$$$$ a+e=a, b+e=b, a^{-1}+e=a^{-1}, b^{-1}+e=b^{-1}.$$ Observe that $R\subseteq E(S,+)$ and $$\sigma_x(a)=x(a+x^{-1})=xa+x(x^{-1}+x^{-1})=xa+xx^{-1}=xa+e.$$

(1) $\delta$ is a right action.  On one hand,
$$(\delta_b\delta_a)(x)=\delta_b(\delta_a (x))=\delta_b((a+x^{-1})^{-1}a)=(b+a^{-1}(a+x^{-1}))^{-1}b.$$
On the other hand,
$$\delta_{ab}(x)=(ab+x^{-1})^{-1}ab=(a^{-1}(ab+x^{-1}))^{-1}b=(a^{-1}ab+a^{-1}(a+x^{-1}))^{-1}b.$$
Thus $\delta_b\delta_a=\delta_{ab}$.  Moreover, $$\delta_e(x)=(e+x^{-1})^{-1}e=(x^{-1})^{-1} e=x.$$ Thus $\delta_a\delta_{a^{-1}}=\delta_{a^{-1}} \delta_a=\delta_e={\rm id}_R$, whence $\delta_a$ is bijective.

(2) $\sigma$ is a left action.   By Lemma \ref{lambdaidempot}, we have $x(x^{-1}+e)=\lambda_x(e)\in E(S,+)$. So
$$\sigma_x\sigma_y(a)=\sigma_x(ya+e)=x(ya+e)+e=xya+x(x^{-1}+e)+e=xya+e=\sigma_{xy}(a)$$ by Lemma \ref{key1} (4). Thus $\sigma_x\sigma_y=\sigma_{xy}$.
Moreover, $\sigma_e(a)=ea+e=a+e=a$. This gives that $\sigma_e={\rm id}_K$, whence $\sigma_{x^{-1}}\sigma_x=\sigma_x\sigma_{x^{-1}}=\sigma_e={\rm id}_K$. So $\sigma_x$  is bijective.

(3) We next consider (a). Since $e+x^{-1}=x^{-1}$, by (\ref{dengshi}) and Lemma \ref{key1} we have $$\sigma_x(a)\sigma_{\delta_a(x)}(b)=x(a+x^{-1}) [(a+x^{-1})^{-1}a(b+a^{-1}(a+x^{-1}))]$$
$$=xa(b+a^{-1}(a+x^{-1}))=x(ab+a(a^{-1}+a^{-1}(a+x^{-1})))$$
$$=x(ab+a(a^{-1}(e+x^{-1})))=x(ab+e+x^{-1})=x(ab+x^{-1})=\sigma_{x}(ab).$$

(4)  Since $e+x^{-1}=x^{-1}$, we have $$\sigma_x(e)=x(e+x^{-1})=xx^{-1}=e.$$ Thus (b) holds.

(5) Now, we consider (c). Since $e+x^{-1}=x^{-1}$, by (\ref{dengshi}) and Lemma \ref{key1} we have
$$\delta_{\sigma_y(a)}(x)\delta_a(y)=[(y(a+y^{-1})+x^{-1})^{-1}y(a+y^{-1})](a+y^{-1})^{-1}a$$
$$=(y(a+y^{-1})+x^{-1})^{-1}ya=(y^{-1}(y(a+y^{-1})+x^{-1}))^{-1}a=(a+y^{-1}+y^{-1}(y+x^{-1}))^{-1}a$$$$=(a+y^{-1}(e+x^{-1}))^{-1}a=(a+y^{-1}x^{-1})a=(a+(xy)^{-1})^{-1}a=\delta_{a}(xy).$$

(6) Since $a+e=a$, we have $$\delta_a(e)=(a+e^{-1})^{-1}a=(a+e)^{-1}a=a^{-1}a=e.$$

(7) Since $(R,+)$ is a right zero semigroup, we have $x+y=y$ and $({ \delta_a (x^{-1}) } )^{-1} + ({ \delta_a (y^{-1}) } )^{-1}=({ \delta_a (y^{-1}) } )^{-1}$. Thus item (e) is true.

(8) By Theorem \ref{main1} and Lemma \ref{rectangular} (4) and the fact that $e\in E(S,+)$, we have the map $\pi_e: S\rightarrow S$. For simplicity, we denote $-s=\pi_e(s)$ for all $s\in S$.
 Then  for all $s,t\in S$ and $x\in E(S,+)$,   we have $-x,s-s,-s+s\in E(S,+)$ and $$s-s+s=s, -s+e=-s=e-s, -(s+t)=-t-s.$$
By Lemma \ref{rectangular} (4) and Lemma \ref{key1},  we can observe that
$$a \cdot\sigma_{ ({ \delta_{a}(x^{-1})})^{-1}}(b)=a\cdot \sigma_{a^{-1}(a+x)}(b)=a\cdot a^{-1} (a+x)(b+(a+x)^{-1}a)$$
$$=(a+x)(b+(a+x)^{-1}a)=(a+x)b+(a+x)((a+x)^{-1}+(a+x)^{-1}a)$$
$$=(a+x)b-x-a+[a+x+(a+x)((a+x)^{-1}+(a+x)^{-1}a)]$$
$$=(a+x)b-x-a+(a+x)(e+(a+x)^{-1}a)$$
$$=(a+x)b-x-a+(a+x)(a+x)^{-1}a   $$
$$ (\mbox{since }(a+x)^{-1}a=\delta_a(x^{-1})\in R=E(e+S))$$
$$=(a+x)b-x-a+a=(a+x)b-x-a+a+e=(a+x)b+e$$
$$(\mbox{Since } -x, -a+a\in E(S,+)).$$
Since $a^{-1}(a+x)=(\delta_a(x^{-1}))^{-1}\in R=E(e+S)$, we have
$$\sigma_{(\delta_a(x^{-1}))}(a^{-1})=\sigma_{(a+x)^{-1}a}(a^{-1})=(a+x)^{-1}a(a^{-1}+a^{-1}(a+x))$$
$$=(a+x)^{-1}aa^{-1}+(a+x)^{-1}a(a^{-1}(a+x)+a^{-1}(a+x))$$
$$=(a+x)^{-1}+(a+x)^{-1}a a^{-1}(a+x)=(a+x)^{-1}+e.$$
This implies that
$$a\cdot \sigma_{(\delta_a(x^{-1}))^{-1}}(b) +a\cdot  \sigma_{(\delta_a(x^{-1}))^{-1}}(\sigma_{(\delta_a(x^{-1}))}(a^{-1})+c)$$
$$=a\cdot \sigma_{(\delta_a(x^{-1}))^{-1}}(b) +a\cdot  \sigma_{(\delta_a(x^{-1}))^{-1}}((a+x)^{-1}+e+c)$$
$$=(a+x)b+e+(a+x)((a+x)^{-1}+c)+e$$$$=(a+x)b+(a+x)((a+x)^{-1}+c)+e$$
$$=(a+x)(b+c)+e=a\cdot \sigma_{(\delta_a(x^{-1}))^{-1}}(b+c).$$
Thus item (f) is true.

By (1)--(8) and Lemma \ref{colazzo}, we have the left semibrace $(K,+, \cdot) \bowtie^\delta_\sigma (R, +, \cdot).$
Define $$\phi:(K,+, \cdot) \bowtie^\delta_\sigma (R, +, \cdot)\rightarrow S,\,\,\, (a,x)\mapsto a+x.$$
Then $\phi$ is a map obviously. We shall show   $\phi$ is a left semibrace isomorphism. Let $x,y\in R$ and $a,b\in K$.

(i)  If $a+x=b+y$, then $$a=a+e=a+x+e=b+y+e=b+e=b$$ by Lemma \ref{key1} (4),  and so $a+x=a+y$. As $-a+a\in R\subseteq E(S,+)$ by Lemma \ref{rectangular}, we have $$x=e+x=e-a+a+x=e-a+a+y=e+y=y$$ by Lemma \ref{key1} (2).
This shows that $\phi$ is injective.

(ii) Let $s\in S$. Then $s+e\in K$ and $-s+s\in R$. This implies that $$\phi(s+e, -s+s)=s+e-s+s=s-s+s=s$$ by Lemma \ref{key1} (2). Thus $\phi$ is surjective.

(iii) $\phi$ preserves $+$. In fact, $$\phi((a,x)+(b,y))=\phi((a+b, x+y))$$$$=a+b+x+y=a+b+y=(a+x)+(b+y)=\phi((a,x))+\phi((b,y)).$$

(iv)  By Lemma \ref{rectangular} (4) and Lemma \ref{key1},  we can observe that $e+y=y, e+x=x$ and
$$ x \cdot ({
\delta_{({ \sigma_{ x^{-1}}(a)})^{-1}}(y^{-1})})^{-1} =x\cdot (((a+x)^{-1}x+y)^{-1}(a+x)^{-1}x)^{-1}$$
$$=xx^{-1} (a+x)((a+x)^{-1}x+y)=(a+x)((a+x)^{-1}x+y)$$
$$=(a+x)(a+x)^{-1}x+(a+x)((a+x)^{-1}+y)=x+(a+x)((a+x)^{-1}+y)$$
$$=x-x-a+a+x+(a+x)((a+x)^{-1}+y)$$$$=x-x-a+(a+x)(e+y)=e+x-x-a+(a+x)(e+y)$$$$=-e-a+(a+x)y=-(a+e)+(a+x)y=-a+(a+x)y.$$
By Lemmas \ref{key1} and \ref{rectangular}, this implies that
$$\phi((a,x)(b,y))=\phi((a+x)b+e, -a+(a+x)y)$$$$=(a+x)b+e-a+(a+x)y=(a+x)b-a+(a+x)y,$$
$$\phi(a,x)\phi(b,y)=(a+x)(b+y)=(a+x)b+  (a+x)((a+x)^{-1}+y)$$$$=(a+x)b-x-a+a+x+  (a+x)((a+x)^{-1}+y)$$
$$=(a+x)b-x-a+(a+x)(e+y)$$$$=(a+x)b-x-a+(a+x)y=(a+x)b-a+(a+x)y.$$
Thus $\phi((a,x)(b,y))=\phi(a,x)\phi(b,y)$.
\end{proof}

\section{Solutions  associated to near left semibraces}
In this section, we consider a new map associated to a  near left semibrace. We give a sufficient and necessary condition under which such a map forms a set-theoretic  solution of the Yang-Baxter equation. This is an analogue of \cite[Theorem 3]{Catino-Colazzo-Stefanelli4}. Some special cases are also considered.
Let $(S,+,\cdot,\mu)$ be a near left semibrace and $a\in S$. Define
$$\lambda_{a} : S \longrightarrow S: b\mapsto a(a^{-1}\mu(e)+b),\,\,\,\,\,\, \rho_a :  S \longrightarrow S : b \mapsto   (b^{-1}\mu(e)+a)^{-1}a,$$
$$\lambda: (S, \cdot) \longrightarrow \textup{Map}(S,S),\, a\mapsto \lambda_a\,\,\,\, \rho: (S, \cdot) \longrightarrow \textup{Map}(S,S),\, a\mapsto \rho_a.$$
Observe by (\ref{near}) and Lemma \ref{near1} (1), (2) that
\begin{equation}\label{jibenjisuan}
\lambda_a(b)=\mu(e)+\mu(a)+ab,\,\, c+\lambda_a(b)=c+\mu(e)+\mu(a)+ab=c+\mu(a)+ab,
\end{equation}
\begin{equation}\label{jibenjisuan1}
a+\lambda_{b^{-1}}\lambda_b(c)=a+\lambda_{b^{-1}b}(c)=a+\lambda_e(c)=a+\mu(e)+\mu(e)+c=a+c,
\end{equation}
for all $a,b,c\in S$.
\begin{lemma}\label{tongtai}
Let $(S,+,\cdot,\mu)$ be a near left semibrace and $a,b,c\in S$. Then $\lambda_a\in \mbox{\rm End}(S,+)$ and $\lambda$ is a semigroup homomorphism.  Moreover, $\lambda_a(b)\rho_b(a)=ab$.
\end{lemma}
\begin{proof}
Let $b,c\in S$. Then by (\ref{jibenjisuan}), (\ref{near}) and Lemma \ref{near1} (1), we have  $$\lambda_a(b+c)=\mu(e)+\mu(a)+a(b+c)=\mu(e)+\mu(a)+ab+\mu(a)+ac$$
$$=\mu(e)+\mu(a)+ab+\mu(e)+\mu(a)+ac =\lambda_a(b)+\lambda_a(c).$$
This gives that $\lambda_a\in \mbox{\rm End}(S,+)$. Moreover,   by (\ref{jibenjisuan}), (\ref{near}) and Lemma \ref{near1} (2),  $$\lambda_a\lambda_b(c)=\mu(e)+\mu(a)+a\lambda_b(c)=\mu(e)+\mu(a)+ab(b^{-1}\mu(e)+c)$$$$=\mu(e)+\mu(a)+a\mu(e)+\mu(ab)+abc=\mu(e)+ \mu(ab)+abc=\lambda_{ab}(c).$$
This shows that $\lambda$ is a semigroup homomorphism. The final assertion is obvious.
\end{proof}
\begin{theorem}\label{main3}
 Let $(S,+,\cdot,\mu)$ be a near left semibrace. The map ${r}_S:S\times S \times \to S\times S$, defined by ${r}_S(a,b)=(\lambda_a(b), \rho_b(a))$ for all $a,b \in S$, is a solution if and only if
\begin{eqnarray*}\label{tiaojian}
&(a\mu(e)+\mu(b)+bc)\mu(e)+\mu(\mu(e)+\mu(b)+bc)+bc\\
&=(a\mu(e)+b)\mu(e)+\mu(b)+bc
\end{eqnarray*}
        for all $a,b,c\in S$.
\end{theorem}
\begin{proof}
It is a routine matter to prove  that ${r}_S$ is  a   solution  of the Yang-Baxter equation if and only if
\begin{equation}\label{jie1}
\lambdaa{a}{\lambdaa{b}{\left(c\right)}} = \lambdaa{\lambdaa{a}{\left(b\right)}}{\lambdaa{\rhoo{b}{\left(a\right)}}{\left(c\right)}},
\end{equation}
\begin{equation}\label{jie2}
\lambdaa{\rhoo{\lambdaa{b}{\left(c\right)}}{\left(a\right)}}{\rhoo{c}{\left(b\right)}} =\rhoo{\lambdaa{\rhoo{b}{\left(a\right)}}{\left(c\right)}}{\lambdaa{a}{\left(b\right)}},
\end{equation}
\begin{equation}\label{jie3}
\rhoo{\rhoo{b}{\left(c\right)}}{\rhoo{\lambdaa{c}{\left(b\right)}}{\left(a\right)}} = \rhoo{c}{\rhoo{b}{\left(a\right)}}
\end{equation}
for all $a,b,c\in S$.  Let $a,b,c\in S$. Then by Lemma \ref{tongtai}, $$\lambdaa{\lambdaa{a}{\left(b\right)}}{\lambdaa{\rhoo{b}{\left(a\right)}}{\left(c\right)}}
                =\lambdaa{\left(\lambdaa{a}{\left(b\right)}\right)\left(\rhoo{b}{\left(a\right)}\right)}{\left(c\right)} = \lambdaa{a b}{\left(c\right)} = \lambdaa{a}{\lambdaa{b}{\left(c\right)}},$$
This shows  (\ref{jie1}) always holds.  Observe that $\lambda_b(c)\rho_c(b)=bc$ by  Lemma \ref{tongtai}. So
\begin{eqnarray*}
 &\lambdaa{\rhoo{\lambdaa{b}{\left(c\right)}}{\left(a\right)}}{\rhoo{c}{\left(b\right)}} =
  \rhoo{\lambdaa{b}{\left(c\right)}}{\left(a\right)}  \left(\left(\rhoo{\lambdaa{b}{\left(c\right)}}{\left(a\right)}\right)^{-1} \mu(e) +\rhoo{c}{\left(b\right)}\right)\\
 &=\left(a^{-1}\mu(e)+\lambdaa{b}{\left(c\right)}\right)^{-1}\cdot  \lambdaa{b}{\left(c\right)} \cdot \left(\left(\lambdaa{b}{\left(c\right)}\right)^{-1} \left(a^{-1}\mu(e)+\lambdaa{b}{\left(c\right)}\right)\mu(e)+\rhoo{c}{\left(b\right)}\right)\\
 &=\left(a^{-1}\mu(e)+\lambdaa{b}{\left(c\right)}\right)^{-1} \cdot \left( (a^{-1}\mu(e)+\lambda_b(c))\mu(e)+\mu(\lambda_b(c))+bc\right)\\
 &=\left(a^{-1}\mu(e)+\lambdaa{b}{\left(c\right)}\right)^{-1} \cdot \left( (a^{-1}\mu(e)+\mu(b)+bc)\mu(e)+\mu(\mu(e)+\mu(b)+bc)+bc\right)
\end{eqnarray*}
by (\ref{near}) and (\ref{jibenjisuan}). On the other hand, by Lemma \ref{tongtai} we have  $\lambda_a(b)\rho_b(a)=ab$ and
\begin{eqnarray*}
&\rhoo{\lambdaa{\rhoo{b}{\left(a\right)}}{\left(c\right)}}{\lambdaa{a}{\left(b\right)}}
=\left(\left(\lambdaa{a}{\left(b\right)}\right)^{-1}\mu(e) + \lambdaa{\rhoo{b}{\left(a\right)}}{\left(c\right)}\right)^{-1}  \lambdaa{\rhoo{b}{\left(a\right)}}{\left(c\right)}\\
&=\left(\left(\lambdaa{a}{\left(b\right)}\right)^{-1}\mu(e) +\mu(\rho_b(a))+\rho_b(a) c
 \right)^{-1} \cdot \rhoo{b}{\left(a\right)} \left(\left(\rhoo{b}{\left(a\right)}\right)^{-1}\mu(e)+c\right)\,\,\,  (\mbox{by (\ref{jibenjisuan})})\\
 &=\left(\rho_b(a) (\rho_b(a))^{-1}\left(\lambdaa{a}{\left(b\right)}\right)^{-1}\mu(e)  +\mu(\rho_b(a))+\rho_b(a) c
 \right)^{-1}\cdot \rhoo{b}{\left(a\right)} \left(\left(\rhoo{b}{\left(a\right)}\right)^{-1}\mu(e)+c\right)\\
 &=\left(\rho_b(a) [(\rho_b(a))^{-1}\left(\lambdaa{a}{\left(b\right)}\right)^{-1}\mu(e)  +  c]
  \right)^{-1} \cdot \rhoo{b}{\left(a\right)} \left(\left(\rhoo{b}{\left(a\right)}\right)^{-1}\mu(e)+c\right)\,\,\,  (\mbox{by (\ref{near})})\\
 &=\left((\lambdaa{a}{\left(b\right)} \rho_b(a))^{-1}\mu(e)  +  c
  \right)^{-1}  (\rho_b(a))^{-1} \rhoo{b}{\left(a\right)} \left(\left(\rhoo{b}{\left(a\right)}\right)^{-1}\mu(e)+c\right)\\
 &=\left([(ab)^{-1}\mu(e)  +  c]
   \right)^{-1}  \left(\left(\rhoo{b}{\left(a\right)}\right)^{-1}\mu(e)+c\right)\\
&=\left(b^{-1}a^{-1}\mu(e)  + \lambda_{b^{-1}}\lambda_b(c)
    \right)^{-1}  \left(\left(\rhoo{b}{\left(a\right)}\right)^{-1}\mu(e)+\lambda_{b^{-1}}\lambda_b(c) \right)\,\,\,  (\mbox{by (\ref{jibenjisuan1})})\\
&=\left(b^{-1}a^{-1}\mu(e)   +\mu(b^{-1})+b^{-1}\lambda_b(c)
    \right)^{-1}  \left(b^{-1}(a^{-1}\mu(e)+b) \mu(e) +\mu(b^{-1})+b^{-1}\lambda_b(c) \right)\,\,\,  (\mbox{by (\ref{jibenjisuan})})\\
&=\left(b^{-1}(a^{-1}\mu(e) +  \lambda_b(c))
    \right)^{-1}  \cdot b^{-1}((a^{-1}\mu(e)+b) \mu(e) +\lambda_b(c))\,\,\,  (\mbox{by (\ref{near})})\\
&=(a^{-1}\mu(e) +  \lambda_b(c))^{-1}  bb^{-1}((a^{-1}\mu(e)+b) \mu(e) +\lambda_b(c))\\
&=(a^{-1}\mu(e) +  \lambda_b(c))^{-1}   ((a^{-1}\mu(e)+b) \mu(e) +\lambda_b(c))\\
&=(a^{-1}\mu(e) +  \lambda_b(c))^{-1}   ((a^{-1}\mu(e)+b) \mu(e) + \mu(b)+bc).\,\,\,  (\mbox{by (\ref{jibenjisuan})})
\end{eqnarray*}
 Finally, we obtain that
\begin{eqnarray*}
&\rhoo{\rhoo{c}{\left(b\right)}}{\rhoo{\lambdaa{b}{\left(c\right)}}{\left(a\right)}}
 =\left(\left(\rhoo{\lambdaa{b}{\left(c\right)}}{\left(a\right)}\right)^{-1}\mu(e)+\rhoo{c}{\left(b\right)}\right)^{-1} \rhoo{c}{\left(b\right)}\\
&=\left(\left(\lambdaa{b}{\left(c\right)}\right)^{-1}  \left(a^{-1}\mu(e)+\lambdaa{b}{\left(c\right)}\right)\mu(e)+\rhoo{c}{\left(b\right)}\right)^{-1} \rhoo{c}{\left(b\right)}\\
&=\left(\left(\lambdaa{b}{\left(c\right)}\right)^{-1}  \left(a^{-1}\mu(e)+\lambdaa{b}{\left(c\right)}\right)\mu(e)+\lambdaa{\left(\lambdaa{b}{\left(c\right)}\right)^{-1}}{\lambdaa{\lambdaa{b}{\left(c\right)}}{\rhoo{c}{\left(b\right)}}}\right)^{-1} \rhoo{c}{\left(b\right)} \,\,\,  (\mbox{by (\ref{jibenjisuan1})})\\
&=\left(\left(\lambdaa{b}{\left(c\right)}\right)^{-1}  \left(a^{-1}\mu(e)+\lambdaa{b}{\left(c\right)}\right)\mu(e)+\mu((\lambda_b(c))^{-1})+(\lambda_b(c))^{-1}
{\lambdaa{\lambdaa{b}{\left(c\right)}}{\rhoo{c}{\left(b\right)}}}\right)^{-1} \rhoo{c}{\left(b\right)} \,\,\,  (\mbox{by (\ref{jibenjisuan})})\\
&=\left(\left(\lambdaa{b}{\left(c\right)}\right)^{-1} ( \left(a^{-1}\mu(e)+\lambdaa{b}{\left(c\right)}\right)\mu(e)+  \lambda_{\lambda_b(c)} \rho_c(b)   )\right)^{-1} \rhoo{c}{\left(b\right)} \,\,\,  (\mbox{by (\ref{near})})\\
&=\left(\left(\lambdaa{b}{\left(c\right)}\right)^{-1} ( \left(a^{-1}\mu(e)+\lambdaa{b}{\left(c\right)}\right)\mu(e)+\mu(\lambda_b(c))+\lambda_b(c) \rho_c(b))\right)^{-1} \rhoo{c}{\left(b\right)} \,\,\,  (\mbox{by (\ref{jibenjisuan})})\\
&=\left( ( \left(a^{-1}\mu(e)+\lambdaa{b}{\left(c\right)}\right)\mu(e)+ \mu(\lambda_b(c))+bc)\right)^{-1} \lambdaa{b}{\left(c\right)}  \rhoo{c}{\left(b\right)}  \,\,\,  (\mbox{by (Lemma \ref{tongtai}})\\
&=\left( ( \left(a^{-1}\mu(e)+\lambdaa{b}{\left(c\right)}\right)\mu(e)+ \mu(\lambda_b(c))+bc)\right)^{-1} bc\,\,\,  (\mbox{by (Lemma \ref{tongtai}})\\
&=\left( ( \left(a^{-1}\mu(e)+\mu(b)+bc\right)\mu(e)+ \mu(\mu(e)+\mu(b)+bc)+bc)\right)^{-1} bc\,\,\,  (\mbox{by (Lemma \ref{tongtai}})
\end{eqnarray*}
and
\begin{eqnarray*}
\rhoo{c}{\rhoo{b}{\left(a\right)}}
&=\left(\left(\rhoo{b}{\left(a\right)}\right)^{-1}\mu(e)+c\right)^{-1}  c =\left(b^{-1}  \left(a^{-1}\mu(e)+b\right)\mu(e)+c\right)^{-1}  c\\
&=\left(b^{-1}  \left(a^{-1}\mu(e)+b\right)\mu(e)+\lambdaa{b^{-1}}{\lambdaa{b}{\left(c\right)}}\right)^{-1}c \,\,\,  (\mbox{by (\ref{jibenjisuan1})})\\
&=\left(b^{-1}  \left(a^{-1}\mu(e)+b\right)\mu(e)+   \mu(b^{-1})+b^{-1}\lambda_b(c))\right)^{-1}c \,\,\,  (\mbox{by (\ref{jibenjisuan})})\\
&=\left(b^{-1}  (\left(a^{-1}\mu(e)+b\right)\mu(e)+ \lambda_b(c))\right)^{-1}c \,\,\,  (\mbox{by (\ref{near})})\\
&=\left(b^{-1}  (\left(a^{-1}\mu(e)+b\right)\mu(e)+ \mu(b)+bc)\right)^{-1}c \,\,\,  (\mbox{by (\ref{jibenjisuan})})\\
&=(\left(a^{-1}\mu(e)+b\right)\mu(e)+ \mu(b)+bc)^{-1}b c.
\end{eqnarray*}
Observe that $\left(S, \cdot\right)$ is a group, it follows that the result holds.
\end{proof}

\begin{coro}\label{brace}
 Let $(S,+,\cdot,\mu)$ be a near left  brace such that $0$ is the identity of the group $(S,+)$. The map ${r}_S:S\times S \times \to S\times S$, defined by ${r}_S(a,b)=(\lambda_a(b), \rho_b(a))$ for all $a,b \in S$, is a solution if and only if
\begin{equation}\label{tiaojiane}
(a\cdot 0 -b\cdot 0 +b\cdot c)\cdot 0 -(-b\cdot 0+b\cdot c)\cdot 0=(a\cdot 0+b)\cdot 0 -b\cdot 0
\end{equation}
        for all $a,b,c\in S$.
\end{coro}
\begin{proof}
By Remark \ref{zhun}, we have $\mu(x)=-x\cdot 0$ and $\mu(e)=0$. Therefore, the results follows from Theorem \ref{main3}.
\end{proof}
\begin{remark}Let $(S,+,\cdot,\mu)$ be a near left semibrace. We observe that $\overline{r}_S$ defined in Theorem \ref{main2} and $r_S$ defined in Theorem \ref{main3}  may  be different  solutions. For example, consider $S=\{1,0\}$ with the following addition $+$ and  multiplication $\cdot$:
$$1+0=1=0+1, 0+0=0=1+1;\,\,\, 1\cdot 1=1=0\cdot 0=1, 1\cdot 0=0\cdot 1=0.$$ Moreover, define a map $\mu: S\to S, 0\mapsto 1, 1\mapsto 0$.
Then one can check that $(S,+,\cdot,\mu)$ is a near left semibrace. In fact, $(S,+,\cdot,\mu)$ is a near left brace and $1$ and $0$ are the identities in the groups $(S,+)$ and $(S,\cdot)$, respectively. A routine calculations tell us that
(\ref{tiaojiane}) holds. Thus  both $\overline{r}_S$ and $r_S$ are solutions. However, $$\overline{r}_S(1,1)=(0,-)\not=(1,-)=r_S(1,1).$$
The relationship between  $\overline{r}_S$ and $r_S$ is a topic that deserves further investigation.
\end{remark}
\begin{remark}
Let $(S,+,\cdot,\mu)$ be a near left  semibrace such that $\mu(e)=e$. This together with Lemma \ref{near1} (1) implies that  the condition given in Theorem \ref{main3} becomes
$$a+e+\mu(b)+bc+\mu(e+\mu(b)+bc)\mu(e)+bc=a+b+\mu(b)+bc.$$
By Lemma \ref{near1} (2), both sides of the above equality are $a+bc$. Thus in this case, the condition given in Theorem \ref{main3} always true, and so  ${r}_S$ defined in Theorem \ref{main3} is always a  solution. In fact, by the fact that $\mu(e)=e$ and  Lemma \ref{near1} (1), we have $x+e+y=x+y$ for all $x,y\in S$. This implies that $(S,+,\cdot)$ is a strong left semibrace. By Theorem \ref{solution}, $\overline{r}_S$ given in Theorem \ref{solution} is a solution. However, in this case,  ${r}_S=\overline{r}_S$ obviously. So we also obtain that ${r}_S$ is always a  solution by   this method.
\end{remark}

\vspace{0.1cm}

\noindent {\bf Acknowledgment:}
The author thanks Professors Marzia Mazzotta, Paola Stefanelli, and Francesco Catino at Universit$\grave{\rm a}$ del Salento  for   pointing out  Proposition 5 in their paper \cite{Catino-Mazzotta-Stefanelli}. By this result, we can simplify the proof of Theorem \ref{main1} greatly.  The paper is supported partially  by the Nature Science Foundations of  China (12271442, 11661082).

\end{document}